\documentclass[a4paper, 11pt,reqno]{amsart}
\title[Polyharmonic Neumann Problem]{Polyharmonic Neumann and mixed boundary value problems in the Heisenberg group $\mathbb H_n$}
\usepackage{amssymb, latexsym}
\usepackage{mathrsfs}
\usepackage{fancyhdr}
\def\H{\mathbb H}
\def\C{\mathbb C}
\def\R{\mathbb R}
\def\ben{\begin{eqnarray*}}
\def\een{\end{eqnarray*}}
\def\beg{\begin{eqnarray}}
\def\eeg{\end{eqnarray}}
\newtheorem{thm}{Theorem}[section]

\newtheorem{lem}[thm]{Lemma}

\author[S. Dubey, A. kumar and M. M. Mishra ]{Shivani Dubey, Ajay Kumar* and Mukund Madhav Mishra}
\address{Department of Mathematics, University of Delhi, Delhi, India}
\email{shvndb@gmail.com}
\address{Department of Mathematics, University of Delhi, Delhi, India}
\email{akumar@maths.du.ac.in}

\address{Department of Mathematics, Hans Raj College, University of Delhi, Delhi, India}
\email{mukund.math@gmail.com}
\thanks{*Corresponding author. Email: akumar@maths.du.ac.in}
\date{}
\numberwithin{equation}{section}
\begin{document}
\begin{abstract}
We study the polyharmonic Neumann and mixed boundary value problems on the Kor\'{a}nyi ball in the Heisenberg group $\H_n$. Necessary and sufficient solvability conditions are obtained for the nonhomogeneous polyharmonic Neumann problem and Neumann-Dirichlet boundary value problems.
\end{abstract}
\maketitle
\section{Introduction}
 As we know, by convolution, the expressions of iterated polyharmonic Green and polyharmonic Neumann functions can not be easily constructed in explicit form because of complicated computation. Rewriting the higher order Poisson equation $\Delta^m u = f$ in a plane domain as a system of Poisson equations it is immediately clear what boundary conditions may be prescribed in order to get (unique) solutions. Neumann conditions for the Poisson equation lead to higher-order Neumann (Neumann-m) problems for $\Delta^m u = f$. Extending the concept of Neumann functions for the Laplacian to Neumann functions for powers of the Laplacian leads to an explicit representation of the solution to the Neumann-m problem for $\Delta^m u = f$. The representation formula provides the tool to treat more general partial differential equations with leading term $\Delta^m u$  in reducing them into some singular integral equations. Integral representations for solutions to higher order differential equations can be obtained by iterating those representations formulas for first order equations.\\
In case of the unit disk $\mathbb D$ of the complex plane $\C$ the higher order Green function is explicitly known, see e.g. \cite{beg2}. Although this is not true for the higher-order Neumann function nevertheless it can be constructed iteratively. A large number of investigations on various boundary value problems in the Heisenberg group have widely been published \cite{bolaug,jeri,kor6}. However, the polyharmonic Dirichlet problem on the Heisenberg group appeared \cite{amm}. After discussing the Neumann problem for Kohn-Laplacian on the Heisenberg Group $\mathbb H_n$ in \cite{sam}, we consider here the polyharmonic Neumann problem with circular data.\\ In the next section of this article, we study the basic terminologies which have laid the foundation for harmonic analysis on the Heisenberg group. In section 3, we discuss the solvability of the polyharmonic Neumann problem  and a representation formula for a solution of the polyharmonic Neumann problem on the Kor\'{a}nyi ball in $\H_n$ is given. Section 4 deals with the polyharmonic mixed boundary value problems like Neumann-Dirichlet and Dirichlet-Neumann on the Kor\'{a}nyi ball in $\H_n$. 
\section{The Heisenberg group $\H_n$ and Horizontal normal vectors}
2.1. The Lie algebra of $\H_n$ and the Kohn-Laplacian. \\The underlying set of the Heisenberg group is $\C^n \times \R$ with multiplication given by 
$$[z,t].[\zeta,s]=[z+\zeta,t+s+2 \Im (z.\bar{\zeta})], \forall \;[z,t],\;[\zeta,s] \in \C^n \times \R.$$ 
With this operation and the usual $C^{\infty}$ structure on $\C^n \times \R(=\R^{2n} \times \R),$ it becomes a Lie group denoted by $\H_n$. 
A basis of left invariant vector fields on $\H_n$ is given by $\{Z_j,\bar{Z}_j,T:1\leq j\leq n\}$ where, 
\begin{eqnarray*}
Z_j &=& \partial_{z_j}+i\bar{z}_j\partial_t ;\\
\bar{Z}_j &=& \partial_{\bar{z}_j}-iz_j\partial_t ;\\
T &=& \partial_t.
\end{eqnarray*}
If we write $z_j=x_j+iy_j$ and define,
\ben
X_j&=& \partial _{x_j}+2y_j \partial _t,\\
Y_j&=& \partial _{y_j}-2x_j \partial _t,
\een
then, $$Z_j=\frac{1}{2}(X_j-iY_j),$$ and $\{X_j, Y_j, T\}$ is a basis.\\
Given a differential operator $D$ on $\H_n$, we say that it is left invariant if it commutes with left translation $L_g$ for all $g \in \H_n$ and rotation invariant if it commutes with rotations $R_\theta$ for all $\theta \in U(\H_n),$ the unitary group $\H_n$. It is said to be homogeneous of degree $\alpha$ if $$D(f(\delta_rg))=r^\alpha Df(\delta_rg),$$ $\delta_r$ being the Heisenberg dilation defined by $\delta_r[z,t]=[rz,r^2t].$ It can be shown that up to a constant multiple, there is a unique left invariant, rotation invariant differential operator that is homogeneous of degree 2 \cite{foko,than}. This unique operator is called the sub-Laplacian or Kohn-laplacian on the Heisenberg group. 
The sublaplacian $\textsl{L}$ on $\mathbb{H}_n$ is explicitly given by
$$\textsl{L}=-\sum_{j=1}^n\left(X_j^2 + Y_j^2\right).$$
Let $L_0$ denote the slightly modified subelliptic operator $-\frac{1}{4}\textsl{L}$.\\The fundamental solution for $L_0$ on  $\mathbb{H}_n$ with pole at identity is given in \cite{fol} as 
$g_e(\xi)=g_e([z,t])=a_0(|z|^4+t^2)^{-\frac{n}{2}},$
where,
$$a_0=2^{n-2}\frac{(\Gamma(\frac{n}{2}))^2}{\pi^{n+1}},$$ 
is the normalization constant and $\xi=[z,t]$. The fundamental solution of $L_0$ with pole at $\eta$ is given by

$$g_\eta(\xi)=g_e(\eta^{-1}\xi).$$
For a function $f$ on $\H_n$, the Kelvin transform is defined as in \cite{kor82} by $$Kf= N^{-2n}foh,$$ where $h$ is the inversion defined as $$h([z,t])=\left[\frac{-z}{|z|^2-it},\frac{-t}{|z|^4+t^2}\right],$$ for $[z,t] \in\H_n\setminus\{e\}$. From \cite[(3.3)]{kori} we have, for $\eta\neq e\in\H_n$
\begin{eqnarray*}
K(g_\eta)=N(\eta)^{-2n}g_{\eta^*},
\end{eqnarray*}
where we wrote $\eta^*$ for $h(\eta)$.\\ An integrable function $f$ on $\H_n$ is called circular if it is invariant under circle action i.e., $\displaystyle{f=\bar{f}=\frac{1}{2 \pi} \int_0^{2 \pi}{f([e^{i \theta}z,t])d\theta}.}$

2.2. Horizontal Normal Vectors. As in \cite{gav2} and \cite{kor83}, we define a singular Riemannian metric $(M_0)$ as follows. On the linear span of $X_j$, $Y_j (1 \leq j \leq n)$ we define an inner product ${\langle .,. \rangle_0}$ by the condition that the vectors $X_j,\; Y_j$ form an orthonormal system. For vectors not in this span we say that they have infinite length. A vector is said to be horizontal if it has finite length. The horizontal gradient $\nabla_0F$ of a function $F$ on $\H_n$ is defined as the unique horizontal vector such that $$\langle \nabla_0F,V \rangle_0=V. F,$$ where $V.F$ is action of $V$ on $F$, for all horizontal vectors $V$. $\nabla_0F$ can be explicitly written as
$$\nabla_0F=\sum_j{\{(X_jF)X_j+(Y_jF)Y_j\}}.$$
 If a hypersurface in $\H_n$ is given as the level set of a smooth function $F$, at any regular point i.e. a point at which $\nabla_0 F\ne 0$ the horizontal normal unit vector will be defined by $$\frac{\partial}{\partial n_0}:=\frac{1}{||\nabla_0F||_0}\nabla_0F.$$ More precisely, this is the horizontal normal pointing outwards for the domain $\{F<0\}$. A point $\zeta$ on the boundary surface is termed ``characteristic" if $\nabla_0 F(\zeta)=0.$ For smooth $F$, the set of characteristic points form a lower dimensional (and hence of measure zero) subset of the boundary.\\ From here onwards the operators $\frac{\partial}{\partial n_0}$ and $L_0$, whenever applied to a function will be with respect to the variable $\xi$ only.

\section{Polyharmonic Neumann Boundary Value Problem } 
In this section we solve the higher-order Neumann problems for $L_0^mu=f$ to get an explicit representation of the solution. Let $S_0$ denotes the set of characteristic points of $B=\{[z,t]\in \H_n:(|z|^4+t^2)^{1/4} \leq 1\}$.\\We define, for each positive integer $k$, the function $N_k(\eta,\xi)$ inductively. Let $N_1(\eta,\xi)=N(\eta,\xi),$ and when $N_{k-1}(\eta,\xi)$ is defined, define
\beg \label{2}
N_k(\eta,\xi)=\int_B{N_{k-1}(\eta,\zeta)N(\zeta,\xi)d\zeta},
\eeg where $N(\eta,\xi)$ is the Neumann function for $B$ given in \cite{sam} as
\beg \label{N}
N(\eta,\xi)&=& \overline{g_\eta}(\xi)+K(\overline{g_\eta})(\xi^{-1})+h(\eta,\xi),
\eeg
\ben
h(\eta,\xi)&=&\sum_{k=1}^\infty\sum_{m=1}^\infty\sum_{j=1}^{n_k}\frac{2n}{m}a_{m;k}C_{\frac{1}{2}(m-2k)}^{(\frac{n}{2}+k,\frac{n}{2}+k)}(t+i|z|^2)Y_{k;j}(z)\\
&& C_m^{(\frac{n}{2}+k, \frac{n}{2}+k)}(t'+i|z'|^2)\overline{Y_{k;j}(z')}+b_0,
\een  $b_0$ is a constant, $n_k,$ is dimension of the space $H_k=H_{k,k}$ where the space $H_{k,l}$ of complex (solid) spherical harmonics of bidegree $(k,l)$ on $\C^n$ consists of all polynomials $P$ in $z_1,\ldots z_n,\;\bar{z}_1,\ldots \bar{z}_n$, homogeneous of degree $k$ in the $z_j's$ and homogeneous of degree $l$ in the $\bar{z}_j's$  satisfying $$\sum_{j=1}^n{\frac{\partial^2}{\partial z_j \partial \bar{z_j}}P}=0.$$ For each $k,l$ the ${Y_{k,l;j}}'s$ form a basis for $H_{k,l}$, where ${Y_{k,l;j}}'s$ have the form 
$$Y_{k,l;j}(z)=\sum_{q=0}^r{c_q|z^*|^{2q}{z_1}^{k-q}{(\bar{z}_1)}^{l-q}},$$ $z=(z_1,z^*)\in\C^n,\; |z^*|^2=|z_2|^2+\ldots+|z_n|^2,\; \text{where},\;r=\min(k,l)\;\text{and}\;c_0,\ldots,c_r$ are constants, $c_0=1$ and ${c_q}'s$ are determined by the relation $$(k-q)(l-q)c_q+(q+1)(n+q-1)c_{q+1}=0,\;0\leq q<r,$$ for detailed study of spherical harmonics refer to \cite{grko}.\\
The only possible points of singularities of integrands in (\ref{2}) are at $\zeta=\eta$ and $\zeta=\xi$. Estimating near the points of singularities we show that integral on the right-hand side of (\ref{2}) is convergent and this kernel is well defined.\\
\begin{lem} \label{N2}
The functions $N_k$ as defined in (\ref{2}) are integrable as functions of $\xi$ and hence are defined finitely a.e.
\end{lem}
\begin{proof}
We first show that $N_1=N$ is integrable over $B$. Using polar coordinates in $\H_1$ and the fact that measure on $B$ is translation invariant, it follows that $g_\eta(\xi)$ is integrable. The second term in (\ref{N}) is integrable as it is equal to $N(\eta)^{-2n}g_{\eta^*}(\xi^{-1})$ and $h(\eta,\xi)$ being continuous function over $B$ so the function $N(\eta,\xi)$ is integrable over $B$.\\ Next we claim that $N_2(\eta,\xi)$ is integrable by writing $N_2(\eta,\xi)$ as sum of nine integrals $I_i\;(1 \leq i \leq 9)$, where
\begin{equation} \label{N2'}
\left.
\begin{split}
&I_1(\zeta)=\int_B{\overline{g_\eta}(\xi)\overline{g_\xi}(\zeta)d\xi},\; I_2(\zeta)=N(\eta)^{-2n}\int_B{\overline{g_{\eta^*}}(\xi^{-1})\overline{g_\xi}(\zeta)d\xi},\\
&I_3(\zeta)=\int_B{h(\eta,\xi)\overline{g_\xi}(\zeta)d\xi},\; I_4(\zeta)=\int_B{\overline{g_\eta}(\xi)N(\xi)^{-2n}\overline{g_{\xi^*}}(\zeta^{-1})d\xi},\\
&I_5(\zeta)=N(\eta)^{-2n}\int_B{\overline{g_{\eta^*}}(\xi^{-1})N(\xi)^{-2n}\overline{g_{\xi^*}}(\zeta^{-1})d\xi},\\
&I_6(\zeta)=\int_B{h(\eta,\xi)N(\xi)^{-2n}\overline{g_{\xi^*}}(\zeta^{-1})d\xi},\\
&I_7(\zeta)=\int_B{\overline{g_\eta}(\xi)h(\xi,\zeta)d\xi},\; I_8(\zeta)=N(\eta)^{-2n}\int_B{\overline{g_{\eta^*}}(\xi^{-1})h(\xi,\zeta)d\xi},\\
&I_9(\zeta)=\int_B{h(\eta,\xi)h(\xi,\zeta)d\xi}.
\end{split}
\right\}
\end{equation} 
$I_1(\zeta)$ exists for almost all $\zeta$ as $\int_B{\overline{g_\xi}(\zeta)}$ is uniformly bounded. Similarly $I_3(\zeta)$ and $I_7(\zeta)$ exist for almost all $\zeta$. Also
$$I_2(\zeta) \leq N(\eta)^{-2n} \sup_{\xi \in B}\overline{g_{\eta^*}}(\xi)\int_B{\overline{g_\xi}(\zeta)d\xi},$$ so exists. Similarly $I_8(\zeta)$ exists for almost all $\zeta$. Next,
$$I_4(\zeta)\leq \sup_{\xi \in B}\overline{g_{\xi^*}}(\zeta)\int_B{\overline{g_\eta}(\xi) N(\xi)^{-2n}d\xi}, $$ and
\ben
\int_B{\overline{g_\eta}(\xi) N(\xi)^{-2n}d\xi} &\leq& \left[\sup_{B \setminus (B_1 \cup B_2)}\overline{g_\eta}(\xi) N(\xi)^{-2n}\right] \nu(B \setminus (B_1 \cup B_2))\\
&&+ \sup_{B_1}\overline{g_\eta}(\xi) \int_{B_1}{ N(\xi)^{-2n}d\xi}+ \sup_{B_2}N(\xi)^{-2n} \int_{B_2}{\overline{g_\eta}(\xi)d\xi},
\een
where $\nu$ is the Haar measure on the Heisenberg group and $B_1$, $B_2$ are disjoint small neighbourhoods of $e$ and $\eta$ respectively. Applying dominated convergence theorem and Fubini's theorem, it can be shown that $\sup_{\xi \in B}\overline{g_{\xi^*}}(\zeta)$ is integrable function of $\zeta$.\\ For $\eta \neq 0$, $\overline{g_{\eta^*}}(\zeta)$ is a bounded function of $\zeta$. So,
\ben
I_5(\zeta) &\leq& N(\eta)^{-2n} \sup_{\zeta\in B}\overline{g_{\eta^*}}(\zeta)\int_B{N(\xi)^{-2n}\overline{g_{\xi^*}}(\zeta)d\xi}\\
&\leq& K_1'\sup_{\xi \in B}\overline{g_{\xi^*}}(\zeta)\int_B{N(\xi)^{-2n}d\xi}\\
&=& K_1 \sup_{\xi \in B}\overline{g_{\xi^*}}(\zeta),
\een
$K_1'$ and $K_1$ being constants independent of $\zeta$. Since $\sup_{\xi \in B}\overline{g_{\xi^*}}(\zeta)$ is integrable function of $\zeta$, so is $I_5$. On same lines, it can be shown that $I_6(\zeta)$ is integrable function of $\zeta$. Since $h(\eta,\xi)$ is continuous function on $B$ so by Fubini's theorem $I_9(\zeta)$ exists for almost all $\zeta$. Having proved integrability of $I_i\; (1 \leq i \leq 9)$, we can say that the integrals exist finitely a.e. and $N_2(\eta,\zeta)$ is integrable function of $\zeta$.\\ We shall, by induction, prove that $N_k(\eta,\zeta)$ exists for $\eta \neq \zeta$ and as a function of $\zeta$ it is integrable over $B$. Assume, we have proved that $N_k$ exists and is integrable over $B$. We have,
\beg \label{N1}
N_{k+1}(\eta,\zeta)&=& \int_B{N_k(\eta,\xi)N_1(\xi,\zeta)d\xi} \nonumber\\
&=& \int_B{N_k(\eta,\xi)\overline{g_\xi}(\zeta)d\xi}+N(\zeta)^{-2n}\int_B{N_k(\eta,\xi)\overline{g_\xi}(-\zeta^*)d\xi}\nonumber\\
&&+\int_B{N_k(\eta,\xi)h(\xi,\zeta)d\xi}.
\eeg
Consider
\ben
\int_B{\int_B{N_k(\eta,\xi)\overline{g_\xi}(\zeta)d\zeta d\xi}} &=& \int_B{N_k(\eta,\xi)\left(\int_B{\overline{g_\xi}(\zeta)d\zeta}\right) d\xi}\\
&=& \left(\int_B{\overline{g_e}(\zeta)d\zeta}\right)\left(\int_B{N_k(\eta,\xi) d\xi}\right).
\een
So $\int_B{N_k(\eta,\xi)\overline{g_\xi}(\zeta)d\xi}$ exists a.e. and is integrable. Since for $\zeta \in B$, $\zeta \notin B $, the second term in (\ref{N1}) can be estimated as
$$N(\zeta)^{-2n}\int_B{N_k(\eta,\xi)\overline{g_\xi}(-\zeta^*)d\xi} \leq N(\zeta)^{-2n}\sup_{\xi \in B}\overline{g_\xi}(-\zeta^*)\int_B{N_k(\eta,\xi)d\xi}.$$ As $N(\zeta)^{-2n}\sup_{\xi \in B}\overline{g_\xi}(-\zeta^*)$ is integrable function of $\zeta$, we obtain the integrability of the second term in (\ref{N1}). Similarly we can estimate the third term in (\ref{N1}). Finally, we obtain the integrability of $N_{k+1}$. 
\end{proof}
\begin{thm} \label{N3}
For continuous circular functions $f$ on $B$ and $g_j$'s, $0\leq j \leq m-1$ on $\partial B$, the polyharmonic Neumann BVP
\begin{equation} \label{1}
\left.
\begin{split}
 L_0^m u&=f \;\;\;\text{in}\;B\\
\frac{\partial}{\partial n_0}(L_0^ju)&=g_j \;\;\text{on}\; \partial B \setminus S_0,\;0\leq j\leq m-1 
\end{split}
\right\}
\end{equation} 
is solvable if and only if
\beg \label{1*}
&&\int_B{\left(\int_B{N_{m-j-1}(\eta,\xi)f(\eta)dv(\eta)}-\sum_{\mu=j+1}^{m-1}\int_{\partial B\setminus S_0}{N_{\mu-j}(\eta,\xi)g_{\mu}(\eta)d\sigma(\eta)}\right)dv(\xi)}\nonumber\\
&&= \int_{\partial B \setminus S_0}{g_j(\xi)d\sigma(\xi)},\;0\leq j\leq m-2,\; m \geq 2, \nonumber\\
&&\int_B{f(\xi)dv(\xi)}=\int_{\partial B\setminus S_0}{g_{m-1}(\xi)d\sigma(\xi)},
\eeg and the circular solution is given by
\beg \label{2*}
u(\xi)=\int_B{N_m(\eta,\xi)f(\eta)dv(\eta)}-\sum_{\mu=0}^{m-1}\int_{\partial B\setminus S_0}{N_{\mu+1}(\eta,\xi)g_{\mu}(\eta)d\sigma(\eta)}.
\eeg
\end{thm}
\begin{proof}
 It is easy to see that, using Lemma \ref{N2}, $\displaystyle{\int_B{N_k(\eta,\xi)f(\eta)dv(\eta)}}$ and $\displaystyle{\int_{\partial B\setminus S_0}{N_k(\eta,\xi)g_\mu(\eta)dv(\eta)}}$ exist for all $1\leq k \leq m$. The proof follows by induction. Assume the result is true for $m=k-1$ and consider (\ref{1}) with $m=k$. Let $L_0u=w$ in $\bar{B}$. Then (\ref{1}) reduces to 
\ben
L_0^{k-1}w&=&f \;\;\text{in}\; B\\
\frac{\partial}{\partial n_0}(L_0^jw)&=&g_{j+1},\;\;\text{on}\; \partial B,\;0\leq j \leq k-2.
\een
From induction hypothesis, the above BVP is solvable if and only if
\beg \label{3}
&&\int_{B}{\left(\int_B{N_{k-j-2}(\eta,\zeta)f(\eta)dv(\eta)}-\sum_{\mu=j+1}^{k-2}\int_{\partial B\setminus S_0}{N_{\mu-j}(\eta,\zeta)g_{\mu+1}(\eta)d\sigma(\eta)}\right)dv(\zeta)}\nonumber\\
&&= \int_{\partial B \setminus S_0}{g_{j+1}(\zeta)d\sigma(\zeta)},\; 0\leq j\leq k-3,\; k \geq 3 \nonumber\\
&&\int_B{f(\zeta)dv(\zeta)}=\int_{\partial B\setminus S_0}{g_{k-1}(\zeta)d\sigma(\zeta)},
\eeg and the circular solution is given by
\beg \label{4}
w(\zeta)=\int_B{N_{k-1}(\eta,\zeta)f(\eta)dv(\eta)}-\sum_{\mu=0}^{k-2}\int_{\partial B\setminus S_0}{N_{\mu+1}(\eta,\zeta)g_{\mu+1}(\eta)d\sigma(\eta)}.
\eeg
 The solution of (\ref{1}) with $m=k$ is then the solution of 
\ben
L_0u&=&w \;\;\text{in}\; B\\
\frac{\partial}{\partial n_0}u&=&g_0 \;\;\text{on}\; \partial B.
\een
Using (\cite{sam}, Theorem 4.1) the above BVP is solvable if and only if 
\beg \label{5}
\int_B{w(\zeta)dv(\zeta)}=\int_{\partial B \setminus S_0}{g_0(\zeta)d\sigma(\zeta)},
\eeg and since $w$ is circular using (\ref{4}), the solution is given by
\beg \label{6}
u(\xi)=\int_B{N(\zeta,\xi)w(\zeta)dv(\zeta)}-\int_{\partial B \setminus S_0}{N(\zeta,\xi)g_0(\zeta)d\sigma(\zeta)}.
\eeg 
Combining the solvability conditions (\ref{3}) with (\ref{5}), substituting the value of $w$ from (\ref{4}) into (\ref{6}), using expression of $N_k(\eta,\xi)$ in (\ref{2}) and Lemma \ref{N2} we obtain the result to be true for $m=k$.
Now, we verify that $u(\xi)$ given by (\ref{2*}) is the solution of the Neumann BVP (\ref{1}) if conditions (\ref{1*}) are satisfied.\\ By using definition of $N_m(\eta,\xi)$, we have $L_0N_m(\eta,\xi)=N_{m-1}(\eta,\xi)$ in $B$. Therefore, by induction $L_0^m N_m(\eta,\xi)=\delta_\eta(\xi)$ in $B$.\\ Now it is obvious that $L_0^m u(\xi)=f(\xi)$ in $B$ with $u$ as in (\ref{2*}). Next, for $0 \leq j \leq m-2,$ $L_0^j u(\xi)$ can be expressed as
\ben
&=&\int_B{N_{m-j}(\eta,\xi)f(\eta)dv(\eta)}-\sum_{\mu=j}^{m-1}\int_{\partial B\setminus S_0}{N_{\mu-j+1}(\eta,\xi)g_{\mu}(\eta)d\sigma(\eta)}\\
&=& \int_B{\left(\int_B {N_{m-j-1}(\eta,\zeta)N(\zeta,\xi) dv(\zeta)}\right)f(\eta)dv(\eta)}\\
&&-\sum_{\mu=j}^{m-1}\int_{\partial B\setminus S_0}{\left(\int_B{N_{\mu-j}(\eta,\zeta)N(\zeta,\xi)dv(\zeta)}\right)g_{\mu}(\eta)d\sigma(\eta)},\\
&=& \int_B{\left(\int_B {N_{m-j-1}(\eta,\zeta)f(\eta)dv(\eta)}-\sum_{\mu=j+1}^{m-1}\int_{\partial B\setminus S_0}{N_{\mu-j}(\eta,\zeta)g_{\mu}(\eta)d\sigma(\eta)}\right)N(\zeta,\xi)dv(\zeta)}.
\een 
Therefore, by (\cite{sam}, Theorem 4.1)
$$\frac{\partial}{\partial n_0}(L_0^ju(\xi))=g_j(\xi) \;\;\text{on}\; \partial B \setminus S_0,\;0\leq j\leq m-2,$$ if first equality of (\ref{1*}) is satisfied.\\ For $j=m-1$,
$\frac{\partial}{\partial n_0}(L_0^ju(\xi))=g_{m-1}(\xi)$ if $\int_B{f(\xi)dv(\xi)}=\int_{\partial B\setminus S_0}{g_{m-1}(\xi)d\sigma(\xi)}.$
\end{proof}
\section{Polyharmonic Mixed Boundary Value Problems}
In this section we consider the mixed BVP arising out of $m$-Neumann and $n$-Dirichlet conditions.\\We define, for each positive integer $k$, the functions $G_k(\eta,\xi)$ and $P_k(\eta,\xi)$ inductively.\\ Let $G_1(\eta,\xi)=\overline{G}(\eta,\xi)$ and $P_1(\eta,\xi)=\overline{P}(\eta,\xi)$ and when $G_{k-1}(\eta,\xi)$ and $P_{k-1}(\eta,\xi)$ are defined, define
\beg \label{8}
G_k(\eta,\xi)=\int_B{G_{k-1}(\eta,\zeta)\overline{G}(\zeta,\xi)dv(\zeta)}
\eeg
\beg \label{9}
P_k(\eta,\xi)=\int_B{P_{k-1}(\eta,\zeta)G(\zeta,\xi)dv(\zeta)},
\eeg
where $G(\eta,\xi)$ is the Green's function for $B$ given by $G(\eta,\xi)=g_\eta(\xi)-K(g_\eta)(\xi^{-1})$ and $P(\eta,\xi)$ is the corresponding Poisson kernel for $B$ given by $P(\eta,\xi)=-\frac{1}{4}\frac{\partial}{\partial n_0} G(\eta,\xi)$ as in \cite{kori}. Similar computations as in Lemma \ref{N2}, show that $G_k(\eta,\xi)$ are well defined.\\ The details of the following lemma can be worked out by showing for any continuous $\phi$ defined on $\partial B$, $\int_{\partial B}{P_k(\eta,\xi)\phi(\xi)d\sigma(\xi)}$ exists.
\begin{lem}
The functions $P_k(\eta,\xi)$ as defined in (\ref{9}) are integrable over $\partial B$ for $k \in \mathbb{N}$, and hence exists finitely a.e.
\end{lem} For $0 \leq k \leq m-1$ and $0 \leq j \leq n-1$, we define
\beg \label{8*}
M_{k,j}(\eta,\xi)=\int_B{N_k(\eta,\zeta)G_j(\zeta,\xi)dv(\zeta)},
\eeg
\beg \label{9*}
S_{k,j}(\eta,\xi)=\int_B{N_k(\eta,\zeta)P_j(\zeta,\xi)dv(\zeta)},
\eeg
where $N_k(\eta,\xi)$ is defined in (\ref{2}). 
The only possible points of singularities of integrands in (\ref{8*}) and (\ref{9*}) are at $\zeta=\eta$ and $\zeta=\xi$. Writing $M_{k,j}$ as combination of six integrals $I_i,\; 1\leq i\leq 6$ of (\ref{N2'}) in Lemma \ref{N2} and estimating near these points, we can show that $M_{k,j}$ are well defined. Further in a similar way one can show that $S_{k,j}$ is convergent and the kernel is well defined.
\subsection{Polyharmonic Neumann-Dirichlet problem}
\begin{thm}
For $m,\;n \geq 1$ and continuous circular functions $f$ defined on $B$, $g_r$'s, $0 \leq r \leq m-1$ and $h_s$'s, $0\leq s \leq n-1$ defined on $\partial B$, the mixed BVP
\begin{equation} \label{7}
\left.
\begin{split}
 L_0^{m+n} u&=f \;\;\;\text{in}\;B\\
\frac{\partial}{\partial n_0}(L_0^r u)&=g_r\;\;\text{on}\; \partial B \setminus S_0,\;0 \leq r \leq m-1\\
L_0^{m+s}u&=h_s \;\;\text{on}\; \partial B,\;0 \leq s \leq n-1
\end{split}
\right\}
\end{equation} is solvable if and only if
\beg \label{10}
&&\int_B\left(\int_B{M_{m-r-1,n}(\eta,\xi)f(\eta)dv(\eta)}+\sum_{s=0}^{n-1}{\int_{\partial B}S_{m-r-1,s+1}(\eta,\xi)h_{n-s-1}(\eta)d\sigma(\eta)}\right. \nonumber\\
&&-\left.\sum_{\mu=r+1}^{m-1}{\int_{\partial B \setminus S_0}{N_{\mu-r}(\eta,\xi)g_{\mu}(\eta)d\sigma(\eta)}}\right)dv(\xi)\nonumber\\
&&=\int_{\partial B\setminus S_0}{g_r(\xi)d\sigma(\xi)},\;\;\;\;\;\; 0 \leq r \leq m-2,\nonumber\\
&&\int_B{\left(\int_B{G_n(\eta,\xi)f(\eta)dv(\eta)}+\sum_{s=0}^{n-1}{\int_{\partial B}{P_{s+1}(\zeta,\eta)h_{n-s-1}(\eta)d\sigma(\eta)}}\right)dv(\xi)}\nonumber\\
&&=\int_{\partial B \setminus S_0}{g_{m-1}(\xi)d\sigma(\xi)}.
\eeg
and the circular solution is given by
\beg \label{11}
u(\xi)&=&\int_B{M_{n,m}(\eta,\xi)f(\eta)dv(\eta)}+\sum_{s=0}^{n-1}{\int_{\partial B}{S_{m,s+1}(\eta,\xi)h_{n-s-1}(\eta)d\sigma(\eta)}}\nonumber\\
&-&\sum_{\mu=0}^{m-1} \int_{\partial B \setminus S_0}{N_{\mu+1}(\eta,\xi)g_{\mu}(\eta)d\sigma(\eta)}.
\eeg
\end{thm}
\begin{proof}
As in Lemma \ref{N2}, we can show that $\displaystyle{\int_B{M_{k,l}(\eta,\xi)f(\eta)dv(\eta)}}$;\\ $\displaystyle{\int_{\partial B}{S_{k,l}(\eta,\xi)h_r(\eta)d\sigma(\eta)}}$ and $\displaystyle{\int_{\partial B\setminus S_0}{N_{k}(\eta,\xi)g_\mu(\eta)d\sigma(\eta)}}$ exist and are defined finitely a.e. We decompose the above mixed boundary value problems into $m$-Neumann and $n$-Dirichlet problems. Let $L_0^m u=w$ in $\overline{B}$. Then above problem reduces into the following system of equations:  
\ben
L_0^m u&=&w\;\; \;\;\;\;\;\text{in}\; B\\
\frac{\partial}{\partial n_0}(L_0^r u)&=&g_r\;\;\;\;\;\;\;\text{on}\; \partial B \setminus S_0,\;0 \leq r \leq m-1\\
\een
By Theorem \ref{N3} the above polyharmonic Neumann problem is solvable if and only if (\ref{1*}) is satisfied for $w$ instead of $f$ and the solution is given by
$$u(\xi)=\int_B{N_m(\xi,\zeta)w(\zeta)dv(\zeta)}-\sum_{\mu=0}^{m-1}\int_{\partial B\setminus S_0}{N_{\mu+1}(\xi,\zeta)g_{\mu}(\zeta)d\sigma(\zeta)}.$$ Next consider  
\ben
L_0^n w&=&f\;\; \;\;\;\;\;\text{in}\; B\\
L_0^{s}w&=&h_{s} \;\;\;\;\;\text{on}\; \partial B,\;\;\;\;\;\;\;\;0 \leq s \leq n-1.
\een
The solution of above polyharmonic Dirichlet problem is given in \cite{amm} as
$$w(\zeta)=\int_B{G_n(\zeta,\eta)f(\eta)dv(\eta)}+\sum_{s=0}^{n-1}{\int_{\partial B}{P_{s+1}(\zeta,\eta)h_{n-s-1}(\eta)d\sigma(\eta)}}.$$
Hence, by substituting the value of $w(\zeta)$ in $u(\xi)$ and in the solvability condition, combining the integrals using Fubini's theorem and (\ref{8*}), (\ref{9*}), we get the required solvability conditions (\ref{10}) and the circular solution is given by (\ref{11}).\\
As we have verified in the last section, one can easily verify that $u(\xi)$ given by (\ref{11}) is the solution of the BVP (\ref{7}) if the solvability condition (\ref{10}) is satisfied.
\end{proof}
\subsection{Polyharmonic Dirichlet-Neumann problem}
\begin{thm}
For $m,\;n \geq 1$ and continuous circular functions $f$ defined on $B$, $g_r$'s, $0 \leq r \leq n-1$ and $h_s$'s, $0\leq s \leq m-1$ defined on $\partial B$, the mixed BVP
\begin{equation} \label{eq:5.28}
\left.
\begin{split}
 L_0^{m+n} u&=f \;\;\;\text{in}\;B\\
L_0^{r}u&=g_r \;\;\text{on}\; \partial B,\;0 \leq r \leq n-1\\
\partial^\perp(L_0^{n+s} u)&=h_s\;\;\text{on}\; \partial B,\;0 \leq s \leq m-1
\end{split}
\right\}
\end{equation} is solvable if and only if
\beg \label{eq:5.29}
&&\int_B\left(\int_B{N_{m-s-1}(\eta,\xi)f(\eta)dv(\eta)}-\sum_{\mu=s+1}^{m-1}{\int_{\partial B}N_{\mu-s}(\eta,\xi)h_{\mu}(\eta)d\sigma(\eta)}\right)dv(\xi) \nonumber\\
&&=\int_{\partial B}{h_s(\xi)d\sigma(\xi)},\;\;\;\;\;\; 0 \leq s \leq m-2,\;m \geq 2,\nonumber\\
&&\int_B{f(\xi)dv(\xi)}=\int_{\partial B}{h_{m-1}(\xi)d\sigma(\xi)},
\eeg
and the circular solution is given by
\beg \label{eq:5.30}
u(\xi)&=&\int_B{M_{m,n}(\eta,\xi)f(\eta)dv(\eta)}-\sum_{\mu=0}^{m-1}{\int_{\partial B}{M_{\mu+1,n}(\eta,\xi)h_{\mu}(\eta)d\sigma(\eta)}}\nonumber\\
&+&\sum_{r=0}^{n-1} \int_{\partial B}{P_{r+1}(\eta,\xi)h_{n-r-1}(\eta)d\sigma(\eta)}.
\eeg
\end{thm}
\begin{proof}
We decompose the above mixed boundary value problems into $m$-Neumann and $n$-Dirichlet problems. Let $L_0^n u=w$ in $\overline{B}$. Then above problem reduces into the following system of equations:  
\begin{equation} \label{eq:5.31}
\left.
\begin{split} 
L_0^n u&=w\;\; \;\;\;\;\;\text{in}\; B\\
L_0^r u&=g_r\;\;\;\;\;\;\;\text{on}\; \partial B ,\;0 \leq r \leq n-1
\end{split}
\right\}
\end{equation} and
\begin{equation} \label{eq:5.32}
\left.
\begin{split} 
L_0^m w&=f\;\; \;\;\;\;\;\text{in}\; B\\
\partial^\perp(L_0^s w)&=h_s\;\;\;\;\;\;\;\text{on}\; \partial B,\;0 \leq s \leq m-1.
\end{split}
\right\}
\end{equation}

The solution of polyharmonic Dirichlet problem (\ref{eq:5.31}) as given in \cite{amm} is
$$u(\xi)=\int_B{G_n(\xi,\zeta)w(\zeta)dv(\zeta)}+\sum_{r=0}^{n-1}{\int_{\partial B}{P_{r+1}(\xi,\zeta)g_{n-r-1}(\zeta)d\sigma(\zeta)}}.$$ By Theorem \ref{N3}, the polyharmonic Neumann problem (\ref{eq:5.32}) is solvable if and only if (\ref{eq:5.29}) is satisfied and the solution is given by
$$w(\zeta)=\int_B{N_m(\zeta,\eta)f(\eta)dv(\eta)}-\sum_{\mu=0}^{m-1}\int_{\partial B}{N_{\mu+1}(\zeta,\eta)h_{\mu}(\eta)d\sigma(\eta)}.$$
Hence, by substituting the value of $w(\zeta)$ in $u(\xi)$, combining the integrals using Fubini's theorem and (\ref{8*}), we get the circular solution as given by (\ref{eq:5.30}).
As we have verified in the last section, one can easily show that $u(\xi)$ given by (\ref{eq:5.30}) is the solution of the BVP (\ref{eq:5.28}) if the solvability condition (\ref{eq:5.29}) is satisfied.
\end{proof}

\section*{acknowledgements}
 The first author is supported by the Senior Research Fellowship of Council of Scientific and Industrial Research, India (Grant no. 09/045(1152)/2012-EMR-I) and the second author is supported by R \& D grant from University of Delhi, Delhi, India.

\end{document}